\documentclass[de, en, bibtex, secnum, autoprintbibfalse]{./cls/myamspaper}

\usepackage[top=40mm, bottom=40mm, left=30mm, right=30mm]{geometry}


\DeclareMathAlphabet{\mathpzc}{OT1}{pzc}{m}{it}

\title{A dynamical proof of the van der Corput inequality}

\begin{document}

\begin{abstract}
  We provide a dynamical proof of the van der Corput inequality for sequences in Hilbert spaces 
  that is based on the Furstenberg correspondence principle. This is done by reducing the inequality to the mean 
  ergodic theorem for contractions on Hilbert spaces.
  The key difficulty therein is that the Furstenberg correspondence principle is, a priori, limited to scalar-valued
  sequences. We therefore discuss how interpreting the Furstenberg correspondence 
  principle via the Gelfand--Naimark--Segal construction for $\uC^*$-algebras allows to study
  not just scalar but general Hilbert space-valued
  sequences in terms of unitary operators. This yields a proof of the
  van der Corput inequality in the spirit of the Furstenberg correspondence principle and
  the flexibility of this method is discussed via new proofs for different variants of the 
  inequality.
\end{abstract}

\maketitle

The \emph{van der Corput inequality} (also known as \emph{van der Corput lemma} or \emph{van der Corput trick}) is a well-known 
and versatile tool in ergodic theory commonly used for complexity reduction since it allows to study the decay of norms 
in terms of the decay of correlations. It can be used to prove van der 
Corput's difference theorem on equidistribution of sequences (see, e.g. \cite[Section 3]{KuNi1974} or \cite{Tao2008}), 
to show that every weakly mixing system is weakly mixing of all orders (see \cite[Theorem 9.31]{EFHN2015}), or 
to obtain convergence to zero in the proofs of weighted, subsequential, polynomial and multiple ergodic theorems as 
well as Wiener--Wintner results (see, e.g., \cite{Kra2007}, \cite[Section 7.4]{EiWa2011}, \cite[Chapter 21]{EFHN2015} 
and \cite{ISEM2018}). Here is a common formulation of the inequality (see, e.g., \cite[Lemma 21.5]{HoKr2018}).

\begin{theorem*}
	For every bounded sequence $(u_n)_{n \in \N}$ in a Hilbert space $H$ the inequality
		\begin{align*}
			\limsup_{N \rightarrow \infty} \left\|\frac{1}{N} \sum_{n=1}^N u_n\right\|^2 
			\leq \limsup_{J \rightarrow \infty} \frac{1}{J}\sum_{j=1}^J \limsup_{N \rightarrow \infty} \frac{1}{N}\sum_{n=1}^{N} \Re(u_n|u_{n+j})
		\end{align*}
	holds.
\end{theorem*}
The standard proof of the inequality rests on an application of the Cauchy--Schwarz inequality as well as some elementary computations 
(see, e.g., the proof of \cite[Lemma 21.5]{HoKr2018}). There also exist more conceptual approaches, for example by means of 
positive definite functions, dilation theory, and the mean ergodic theorem (see \cite[Theorem 2.12]{BeMo2016} and the blog post 
\cite{More2015}). This article provides a new perspective on the inequality in terms of the Furstenberg correspondence principle.

Given a bounded sequence of scalars, 
the Furstenberg correspondence principle provides a universal way to study the statistical properties
of the sequence via dynamical systems. It does so by constructing Hilbert spaces and unitary operators
that encode the statistical behavior of the sequence in dynamical terms.
Below we show that one can use the strategy laid out by the correspondence principle 
to readily derive the \emph{scalar} van der Corput inequality.%
\footnote{Related but differing approaches are also present in the blog posts \cite{Tao2014} 
by Tao and \cite{More2018} by Moreira in the context of the van der Corput lemma.} Thereby, the scalar van der Corput inequality reduces to the mean ergodic 
theorem and a simple application of the Cauchy--Schwarz inequality.

The interest of reproving such a well-known inequality with a less elementary proof is twofold: 
on the one hand, such a conceptual proof provides a simple way to derive and understand not only 
the van der Corput inequality but also similar inequalities.
More importantly, however, only the scalar-valued
case can be treated in this way since, classically, the Furstenberg correspondence principle only concerns sequences of numbers. Given 
the importance of the correspondence principle as a bridge between number theory and ergodic theory, this begs the question: 
Does there exist some version of the correspondence principle
for Hilbert space-valued sequences? We address this question
in the first two sections by showing that, also in the case of Hilbert space-valued sequences,
it is possible to construct a Hilbert space and a unitary operator acting on it which encode 
properties of the sequence. We show that this allows, in complete 
analogy to the scalar case, to reduce the general Hilbert space-valued van der Corput inequality 
to the mean ergodic theorem and the Cauchy--Schwarz inequality. Since the key tool in the proof 
is the Gelfand--Naimark--Segal construction for $\uC^*$-algebras, the remaining sections discuss 
how these techniques allow to derive several other van der Corput-type inequalities on $\uC^*$-algebras 
and Hilbert-$\uC^*$-modules. This leads to new proofs of different uniform versions of the 
van der Corput inequality, see \cref{unif1}, \cref{unif2} and \cref{unif3}.

\textbf{Organization of the article.} Section 1 is a short recall of the correspondence principle
and explains the difficulty of extending it to the vector-valued setting. Section 2 then presents all key ideas by giving 
a short dynamical proof of the scalar van der Corput inequality in \cref{vdc1} and then extending these ideas 
to a proof of the general van der Corput inequality in \cref{vdchilbertvalued}. We generalize these
methods to $\uC^*$-algebras in Section 3 and to Hilbert modules in Section 4. In Section 5 we show 
how these extensions allow to recover different uniform versions of the van der Corput inequality. 
The article concludes with a discussion on extensions to F{\o}lner nets and more general groups in 
Section 6.

\section{Motivation: The Furstenberg correspondence principle}\label{sec:motivation}

Given a bounded sequence $(a_n)_n \in \ell^\infty(\N)$ of scalars, the 
Furstenberg correspondence principle provides a way to encode the sequence in a
dynamical system. Classically, this is done by considering the closure 
$A \defeq \overline{\{a_n \mid n\in\N\}}$ of its range,
forming the symbolic shift system $(A^\N, \tau)$ with
$\tau((x_n)_n) = (x_{n+1})_n$, and passing to the subshift $(K; \tau)$ generated by the 
sequence $(a_n)_n$, i.e., $K \defeq \overline{\{ \tau^k((a_n)_n) \mid k\in\N_0\}}$.
(Note that $A^\N$ and hence $K$ is compact since $A$ is a closed, bounded subset of $\C$.)
For the study of statistical properties of the sequence $(a_n)_n$, i.e., the asymptotic
behavior of averages, one can then pass from the topological system $(K; \tau)$ to the 
measure-preserving system $(K, \mathscr{B}, \mu; \tau)$ where $\mathscr{B}$ denotes the 
Borel $\sigma$-algebra of $K$ and $\mu$ is some $\tau$-invariant 
probability measure that represents, e.g., density properties of the original sequence $(a_n)_n$. 
Thus, the problem of understanding statistical properties of the sequence $(a_n)_n$ can be 
translated into the study of the system 
$(K, \mathscr{B}, \mu; \tau)$ and the Koopman operator 
$T_\tau \colon \uL^2(K, \mathscr{B}, \mu) \to \uL^2(K, \mathscr{B}, \mu)$, $f \mapsto f \circ \tau$. This approach has 
proven to be a critical interface between number theory and ergodic theory with applications
ranging from Furstenberg's celebrated proof of Szemerédi's theorem (see \cite{Furs1977}) to
countless further applications and developments such as \cite{MRR2019}, \cite{TaoTer2019}, 
\cite{Fran2017}, or \cite{FranHostKra2013}.

Unfortunately, this approach via symbolic dynamics must inevitably fail for 
problems involving sequences $(u_n)_n \in \ell^\infty(\N; H)$ in infinite-dimensional Hilbert 
spaces $H$, for lack of the 
compactness of the unit ball. To overcome this obstacle, 
we recall an alternative approach to the correspondence principle commonly used in Ramsey
theory that Furstenberg himself was already aware of, see \cite[p.\ 50]{DundasSkau2020}: Using the Stone--\v{C}ech compactification $\beta\N$ of the natural numbers and the 
identification 
\begin{equation*}
  \ell^\infty(\N) \cong \uC(\beta\N),
\end{equation*}
every sequence $(a_n)_n$ corresponds uniquely to its continuous extension $a\in\uC(\beta\N)$.
Since the map $\tau\colon \N \to \N$, $x \mapsto x + 1$ extends continuously to $\beta\N$,
the sequence $(a_n)_n$ is encoded in the topological dynamical system $(\beta\N, \tau)$ via 
$a_n = a(\tau^{n-1}(1))$. The 
invariant measures $\mu$ of this system can be identified with the Banach limits on $\ell^\infty(\N) \cong C(\beta\N)$,
i.e., the shift-invariant positive linear forms, and one can again study the sequence $(a_n)_n$ in terms 
of the dynamical system $(\beta\N, \mathscr{B}, \mu; \tau)$ and its Koopman operator $T_\tau\colon \uL^2(\beta\N, \mathscr{B}, \mu) 
\to \uL^2(\beta\N, \mathscr{B}, \mu)$.

This approach does not immediately extend either to vector-valued sequences $(u_n)_n$ in some Hilbert space $H$
since the isomorphism $\ell^\infty(\N; H) \cong \uC(\beta\N; H)$ holds if and only if $H$ is finite-dimensional. However, a conceptual view of $\uL^2(\beta\N, \mathscr{B}, \mu)$ is to regard
it as the Hilbert space $\uC(\beta\N)^{(\cdot|\cdot)_\mu}$ obtained from the $\uC^*$-algebra $\uC(\beta\N)$ 
by means of the Gelfand--Naimark--Segal construction for the state $\mu \in C(\beta\N)'$. Thus,  
by constructing Hilbert spaces $\mathcal{H}_\mu$ out of $\ell^\infty(\N; H)$, we may achieve a generalization 
of the Hilbert spaces $\uL^2(\beta\N, \mathscr{B}, \mu) \cong \ell^\infty(\N)^{(\cdot|\cdot)_\mu}$ 
to the Hilbert space-valued setting. As we show below, this can be done in terms of a more general GNS construction for 
\emph{Hilbert modules} instead of $\uC^*$-algebras. We carry these arguments out in fairly high generality
to underline that the GNS construction is the key tool needed. However, we start with the most 
basic example in the next section and then develop the general ideas in the following sections.

\section{Van der Corput via the correspondence principle} \label{sec:basic}

We start with a short proof of the scalar van der Corput inequality based on the correspondence principle, 
followed by a second proof in the general case that adapts this idea to the vector-valued setting.

\begin{theorem}\label{vdc1}
	For every bounded sequence $(u_n)_{n \in \N}$ of complex numbers the inequality
		\begin{align*}
			\limsup_{N \rightarrow \infty} \left|\frac{1}{N} \sum_{n=1}^N u_n\right|^2 \leq 
			\liminf_{J \rightarrow \infty} \frac{1}{J}\sum_{j=1}^J \limsup_{N \rightarrow \infty} 
			\frac{1}{N}\sum_{n=1}^{N} \Re (u_n \overline{u_{n+j}})
		\end{align*}
	holds.
\end{theorem}
\begin{proof}
 Let $(u_n)_n\in \ell^\infty(\N)$ and $p\in\beta\N\setminus\N$ be a nonprincipal ultrafilter
 such that 
 \begin{equation*}
  \limsup_{N\to\infty} \left| \frac{1}{N} \sum_{n=1}^N u_n \right|^2 = \lim_{N \to p} \left| \frac{1}{N} \sum_{n=1}^N u_n \right|^2.
 \end{equation*}
 Denote by $u\in\uC(\beta\N)$ the canonical extension of $(u_n)_n$ to $\beta\N$ and for 
 each $n\in\N$ let $\delta_n\in \uC(\beta\N)'$ be the Dirac measure at $n$. Then
 \begin{equation*}
  \lim_{N\to p} \frac{1}{N} \sum_{n=1}^N u_n =
  \lim_{N\to p} \left\langle u, \frac{1}{N}\sum_{n=1}^N \delta_n \right\rangle =
  \left\langle u, \underbrace{\lim_{N\to p}\frac{1}{N}\sum_{n=1}^N \delta_n}_{\eqdef \mu} \right\rangle =
  \int_{\beta\N} u \dmu.
 \end{equation*}
 Here, the probability measure $\mu\in\uC(\beta\N)'$ exists as a weak*-limit because 
 any limit along an ultrafilter exists in a compact space. Now, let 
 $T_\tau\colon \uL^2(\beta\N, \mathscr{B}, \mu)\to\uL^2(\beta\N, \mathscr{B}, \mu)$ be the Koopman operator for 
 the   right   shift $\tau\colon\beta\N \to \beta\N$ and let $P_\mu$ denote its mean ergodic projection.
 By the Cauchy--Schwarz inequality and since $P_\mu$ is a conditional expectation, we
 conclude that 
 \begin{align*}
  \limsup_{N\to\infty} \left| \frac{1}{N} \sum_{n=1}^N u_n \right|^2 
  &= \left| \int_{\beta\N} u \dmu \right|^2
  = \left| \int_{\beta\N} P_\mu u \dmu \right|^2 
  \leq  \int_{\beta\N} |P_\mu u|^2 \dmu \\
  &= \int_{\beta\N} P_\mu u P_\mu \overline{u} \dmu 
  = \int_{\beta\N} P_\mu( u P_\mu \overline{u}) \dmu
  = \int_{\beta\N} u P_\mu \overline{u} \dmu \\
  &= \lim_{J\to\infty} \frac{1}{J} \sum_{j=1}^J \int_{\beta\N} u T_\tau^j \overline{u} \dmu \\
  &= \lim_{J\to\infty} \frac{1}{J} \sum_{j=1}^J \lim_{N\to p} \frac{1}{N} \sum_{n=1}^N u_n \overline{u_{n+j}} \\
  &\leq \liminf_{J \rightarrow \infty} \frac{1}{J}\sum_{j=1}^J \limsup_{N \rightarrow \infty} 
  \frac{1}{N}\sum_{n=1}^{N} \Re (u_n \overline{u_{n+j}}).
 \end{align*}
\end{proof}

To extend this idea to the general Hilbert space case, let $\mu\in\uC(\beta\N)'$ be a 
shift-invariant probability measure as in the proof above. As described in \cref{sec:motivation}, $\mu$ is 
a \emph{state} on the $\uC^*$-algebra $\uC(\beta\N)$, i.e., a positive linear form of norm 1. As such, $\mu$ 
gives rise to a sesquilinear form $(\cdot|\cdot)_\mu$ which yields the Hilbert space $\uL^2(\beta\N, \mathscr{B}, \mu)$
by means of the Gelfand--Naimark--Segal construction.
Now, let $H$ be a Hilbert space and consider the space $\ell^\infty(\N, H)$ of bounded 
sequences in $H$. This space is no longer a $\uC^*$-algebra but still a \emph{$\uC^*$-Hilbert module}
by means of the $\ell^\infty(\N)$-valued inner product 
\begin{align*}
 (\cdot| \cdot)\colon \ell^\infty(\N, H)\times \ell^\infty(\N, H) \to \ell^\infty(\N), \quad 
 ((u_n)_n, (v_n)_n) \mapsto \left((u_n|v_n)_H\right)_n.
\end{align*}
(We will introduce Hilbert modules formally in \cref{sec:modules}.)
Next, we can consider the sesquilinear form 
\begin{equation*}
 (\cdot|\cdot)_\mu \colon \ell^\infty(\N, H)\times \ell^\infty(\N, H) \to \C, \quad 
 ((u_n)_n, (v_n)_n) \mapsto \mu\big(((u_n)_n|(v_n)_n)\big).
\end{equation*}
Since this sesquilinear form is positive semi-definite (which also implies that it is hermitian), it induces an 
inner product on the quotient 
\begin{equation*}
  \ell^\infty(\N, H)/\{x\in \ell^\infty(\N, H) \mid (x|x)_\mu = 0\}.
\end{equation*}
We denote its completion by $\mathcal{H}_\mu$. Note here 
that $\{x\in \ell^\infty(\N, H) \mid (x|x)_\mu = 0\}$ is actually a
subspace of $\ell^\infty(\N, H)$ by the Cauchy-Schwarz inequality for positive semi-definite sesquilinear forms (see, e.g., \cite[Inequality 1.4]{Conw1985}).
 If $\mu = \lim_{N\to p}\frac{1}{N}\sum_{n=1}^N \delta_n$ for some ultrafilter $p\in\beta\N\setminus\N$,   it can be shown that $\mathcal{H}_\mu$ is isomorphic to the quotient
\begin{equation*}
 \left\{ u\colon \N \to H \mmid \limsup_{N \to p} \frac{1}{N} \sum_{n=1}^N \|u_n\|^2 < +\infty \right\}
 /\left\{ u\colon \N \to H \mmid \limsup_{N \to p} \frac{1}{N} \sum_{n=1}^N \|u_n\|^2 = 0\right\}
\end{equation*}
and that under this identification the scalar product takes the form 
\begin{equation*}
 \big([(u_n)_n] | [(v_n)_n]\big)_\mu = \lim_{N\to p} \frac{1}{N}\sum_{n=1}^N (u_n|v_n)_H.
\end{equation*}
Related spaces also play an essential role in \cite{MRR2019} to capture asymptotics of sequences along
F{\o}lner sequences (see in particular Remark 3.2 for when their spaces are complete),
but we will not need this representation here. The nontrivial part of proving this representation is 
to show that the above-defined space is complete; the reader can find the relevant ideas in \cite[Section II.2]{BoFo1944}.
The only thing we will need below is that the shift $T\colon \ell^\infty(\N, H) \to \ell^\infty(\N, H)$,
$(u_n)_n \mapsto (u_{n+1})_n$ induces a contraction on $\mathcal{H}_\mu$. To see this, denote by 
$S\colon \ell^\infty(\N) \to \ell^\infty(\N)$ the shift $(u_n) \mapsto (u_{n+1})_n$ and observe that
for $x, y\in \ell^\infty(\N, H)$, $(Tx|Ty) = S(x|y)$ and so $(Tx|Ty)_\mu = (x|y)_\mu$ by $S$-invariance
of $\mu$. Hence, if $x\in \ell^\infty(\N, H)$ is such that 
$(x|x)_\mu = 0$, then also $(Tx|Tx)_\mu = 0$. This shows that $T$ induces an isometry on $\mathcal{H}_\mu$
which we denote by $T_\mu$.\footnote{In 
fact, it is not hard to show that the opposite shift also induces an isometry on $\mathcal{H}_\mu$ that is 
inverse to $T_\mu$. Hence, $T_\mu\colon \mathcal{H}_\mu \to \mathcal{H}_\mu$ is in fact unitary.}
With this set-up, we are able to prove the full van der Corput inequality. The only notable change is 
that we interpret  the left-hand side in terms of an invariant sesqilinear form instead of an invariant linear form.

\begin{theorem}\label{vdchilbertvalued}
	For every bounded sequence $(u_n)_{n \in \N}$ in a Hilbert space $H$ the inequality
		\begin{align*}
			\limsup_{N \rightarrow \infty} \left\|\frac{1}{N} \sum_{n=1}^N u_n\right\|^2 
			\leq \liminf_{J \rightarrow \infty} \frac{1}{J}\sum_{j=1}^J 
			\limsup_{N \rightarrow \infty} \frac{1}{N}\sum_{n=1}^{N} \Re(u_n|u_{n+j})
		\end{align*}
	holds.
\end{theorem}
\begin{proof}
 Let $u = (u_n)_n \in \ell^\infty(\N, H)$ be a bounded sequence and let $p\in\beta\N \setminus \N$
 be a nonprincipal ultrafilter such that 
 \begin{equation*}
  \limsup_{N \rightarrow \infty} \left\|\frac{1}{N} \sum_{n=1}^N u_n\right\|^2  
  = \lim_{N\to p} \left\|\frac{1}{N} \sum_{n=1}^N u_n\right\|^2
 \end{equation*}
 and define a shift-invariant state $\mu$ on $\ell^\infty(\N)$ via 
 \begin{equation*}
  \mu \defeq \lim_{N\to p} \frac{1}{N} \sum_{n=1}^N \delta_n.
 \end{equation*}
 Consider the sesquilinear form 
 \begin{equation*}
  B\colon\ell^\infty(\N, H)\times\ell^\infty(\N, H) \to \C, \quad 
  (v, w) \mapsto \lim_{N\to p} \left( \frac{1}{N}\sum_{n=1}^N v_n \mmid \frac{1}{N}\sum_{n=1}^N w_n \right)
 \end{equation*}
 and observe that 
 \begin{equation*}
  B(u, u) = \limsup_{N \rightarrow \infty} \left\|\frac{1}{N} \sum_{n=1}^N u_n\right\|^2.
 \end{equation*}
   We show that   $B$ induces a contractive sesquilinear form on $\mathcal{H}_\mu$, i.e.,  
 \begin{equation}\label{eq:contract}
  |B(v, w)| \leq \sqrt{(v|v)_\mu} \sqrt{(w|w)_\mu} = \|[v]\|_{\mathcal{H}_\mu} \|[w]\|_{\mathcal{H}_\mu}.
 \end{equation}
 for all $v, w\in\ell^\infty(\N, H)$.
   Take   $v, w\in\ell^\infty(\N, H)$ and note that for $N\in\N$ we can apply the Cauchy--Schwarz inequality 
 both in $H$ and in $\C^N$ to obtain
 \begin{align*}
  \left|\left( \frac{1}{N} \sum_{n=1}^N v_n \mmid \frac{1}{N}\sum_{n=1}^N w_n\right)\right| 
  &\leq \left\| \frac{1}{N}\sum_{n=1}^N v_n\right\| \left\| \frac{1}{N}\sum_{n=1}^N w_n\right\| 
  \leq \frac{1}{N} \sum_{n=1}^N \|v_n\| \frac{1}{N} \sum_{n=1}^N \|w_n\| \\
  &\leq \left(\frac{1}{N}\sum_{n=1}^N  \|v_n\|^2 \right)^{\frac{1}{2}}\left(\frac{1}{N}\sum_{n=1}^N \|w_n\|^2\right)^{\frac{1}{2}}.
 \end{align*}
 Taking the limit as $N \to p$ proves \eqref{eq:contract}. Therefore, $B$ induces a contractive sesquilinear
 form on $\mathcal{H}_\mu$ which we denote by $B_\mu$. Since for any $n, m\in\N$ and 
 $v, w\in \ell^\infty(\N, H)$ one has $B(S^nv, S^mw) = B(v, w)$, $B_\mu$ satisfies an analogous 
 invariance on $\mathcal{H}_\mu$. Denoting by $P_\mu$ the mean ergodic projection of $T_\mu$ acting on $\mathcal{H}_\mu$
 and using this invariance,
 we conclude that
 \begin{align*}
  \limsup_{N \rightarrow \infty} \left\|\frac{1}{N} \sum_{n=1}^N u_n\right\|^2   
  &= B_\mu([u], [u]) 
  = B_\mu(P_\mu [u], P_\mu [u])
  \leq \|P_\mu [u]\|_\mathcal{H_\mu}\|P_\mu [u]\|_{\mathcal{H_\mu}} \\
  &= ([u]|P_\mu [u])_{\mathcal{H}_\mu}
  = \lim_{J \to \infty} \frac{1}{J} \sum_{j=1}^J ([u]| S^j[u])_{\mathcal{H}_\mu} \\
  &= \lim_{J \to \infty} \frac{1}{J} \sum_{j=1}^J \lim_{N\to p} \frac{1}{N} \sum_{n=1}^N (u_n|u_{n+j}) \\ 
  &\leq \liminf_{J \rightarrow \infty} \frac{1}{J}\sum_{j=1}^J 
			\limsup_{N \rightarrow \infty} \frac{1}{N}\sum_{n=1}^{N} \Re(u_n|u_{n+j}).
 \end{align*}
\end{proof}

Having seen the key ideas in these simple cases, the following sections exploit 
that only very little is actually used in the proofs, allowing to cover generalizations in 
different directions and several versions of the van der Corput inequality. We believe that 
the method presented here could also find applications to other problems but this is beyond the scope of this article.

\section{Van der Corput on C*-algebras}

Since the above proof of the van der Corput inequality depends 
mostly on the GNS construction, we generalize it to arbitrary $\uC^*$-algebras in this 
section. To explain our approach, we start again with the scalar case of the original result.
\begin{theorem}
	For every bounded sequence $(u_n)_{n \in \N}$ of complex numbers the inequality
		\begin{align*}
			\limsup_{N \rightarrow \infty} \left|\frac{1}{N} \sum_{n=1}^N u_n\right|^2 \leq 
			\liminf_{J \rightarrow \infty} \frac{1}{J}\sum_{j=1}^J 
			\limsup_{N \rightarrow \infty} \frac{1}{N}\sum_{n=1}^{N} \Re (u_n \overline{u_{n+j}})
		\end{align*}
	holds.
\end{theorem}
In order to reformulate this in the language of $\uC^*$-algebras, we consider the 
$\uC^*$-algebra $\ell^\infty(\N)$ of all bounded complex sequences and the shift operator 
$S(v_n)_{n \in \N} \coloneqq (v_{n+1})_{n \in \N}$ for $(v_n)_{n \in \N} \in \ell^\infty(\N)$. 
The shift $S \in \mathscr{L}(\ell^\infty(\N))$ is a \emph{Markov operator} meaning that it is
	\begin{enumerate}[(i)]
		\item \emph{unital}, i.e., $S\mathbbm{1} = \mathbbm{1}$, and
		\item \emph{positive}, i.e., $Sx \geq 0$ for every $x \in  \ell^\infty(\N)$ with $x \geq 0$.
	\end{enumerate}
With the state defined by the point evaluation 
$\mu \coloneqq \delta_1 \colon \ell^\infty(\N) \rightarrow \C, \, (v_n)_{n \in \N} \mapsto v_1$ 
and $x \coloneqq (u_n)_{n \in \N} \in \ell^\infty(\N)$ we can rewrite the left-hand side of the desired inequality as 
	\begin{align*}
		\limsup_{N \rightarrow \infty} \left|\frac{1}{N} \sum_{n=1}^N u_n\right|^2  = \limsup_{N \rightarrow \infty} \left|\left\langle \frac{1}{N}\sum_{n=0}^{N-1}S^n x, \nu \right\rangle\right|^2.
	\end{align*}
Likewise, we obtain
	\begin{align*}
		\limsup_{J \rightarrow \infty} \frac{1}{J}\sum_{j=1}^J 
    \limsup_{N \rightarrow \infty} \frac{1}{N}\sum_{n=1}^{N} \Re u_n \overline{u_{n+j}}= 
		\limsup_{J \rightarrow \infty} \frac{1}{J}\sum_{j=1}^{J} 
		\limsup_{N \rightarrow \infty} \frac{1}{N} \sum_{n=0}^{N-1} \Re\langle S^n((S^jx)^* x), \mu\rangle
	\end{align*}
for the right-hand side.

With these considerations, we can obtain \cref{vdc1} by proving that 
for every Markov operator $S \in \mathscr{L}(A)$ on a unital commutative $\mathrm{C}^*$-algebra $A$ 
and every state $\mu \in A'$ one has
	\begin{align}\label{ineq1}
		\limsup_{N \rightarrow \infty} \left|\left\langle \frac{1}{N}\sum_{n=0}^{N-1}S^n x, \mu \right\rangle\right|^2 \leq 
		\liminf_{J \rightarrow \infty} \frac{1}{J}\sum_{j=1}^{J} 
		\limsup_{N \rightarrow \infty} \frac{1}{N} \sum_{n=0}^{N-1} \Re \langle S^n((S^jx)^* x), \mu\rangle.
	\end{align}
In fact, our methods even allow to extend inequality (\ref{ineq1}) to non-commutative $\mathrm{C}^*$-algebras. 

The following operators are the analogue of Markov operators in the non-commutative situation.
Their only essential property for the following is that they define contractions on all induced GNS Hilbert spaces.
\begin{definition}
	A bounded operator $S \in \mathscr{L}(A)$ on a unital $\mathrm{C}^*$-algebra $A$ is a \emph{Markov--Schwarz operator} if 
	\begin{enumerate}[(i)]
		\item it is \emph{unital}, i.e., $S\mathbbm{1} = \mathbbm{1}$, and
		\item it satisfies the \emph{Schwarz inequality}, i.e., $(Sx)^* Sx \leq S(x^*x)$ for all $x \in A$.
	\end{enumerate}
\end{definition}
Every unital $2$-positive operator is Markov--Schwarz. In particular, every unital $^*$-homomorphism is a 
Markov--Schwarz operator. If $A$ is commutative, then the concepts of Markov and Markov--Schwarz operators coincide. We refer to \cite[Chapters 1 and 2]{Stor2013} for an introduction to positive and Schwarz operators on $\mathrm{C}^*$-algebras. 

We show the following version of inequality (\ref{ineq1}) for Markov--Schwarz operators by employing the GNS-construction for $\mathrm{C}^*$-algebras and the mean ergodic theorem for contractions on Hilbert spaces. Here and in the following, we denote by $\mathrm{S}(A) \subset A'$ the set of states on a unital $\mathrm{C}^*$-algebra $A$. This is a convex subset of $A'$ and, equipped with the weak* topology, a compact space (see, e.g., \cite[Subsection 3.2.1]{Pede2018}).
	\begin{theorem}\label{vdc2}
		Let $S \in \mathscr{L}(A)$ be a Markov--Schwarz operator on a unital $\mathrm{C}^*$-algebra $A$. Moreover, let $x \in A$, $(N_i)_{i \in I}$ a subnet of $\N$ and $\mu_i \in \mathrm{S}(A)$ a state on $A$ for every $i \in I$. Then
			\begin{align*}
				\limsup_{i} \left|\left\langle \frac{1}{N_i} \sum_{n=0}^{N_i-1} S^n x, \mu_i\right\rangle\right|^2 \leq \liminf_{J \rightarrow \infty} \frac{1}{J}\sum_{j=1}^{J} \limsup_{i} \frac{1}{N_i} \sum_{n=0}^{N_i-1} \Re \langle S^n((S^jx)^* x), \mu_i\rangle.
			\end{align*}
	\end{theorem}
	\begin{proof}
	For the sake of convenience we write $C_N \coloneqq \frac{1}{N} \sum_{n=0}^{N-1} S^n \in \mathscr{L}(A)$ for every $N \in \N$. Passing to a subnet of $(C_{N_i}'\mu_{N_i})_{i \in I}$ we may assume that 
		\begin{align*}
			\limsup_{i \in I} |\langle C_{N_i}x,\mu_{N_i} \rangle|^2 = \lim_{i \in I} |\langle C_{N_i}x,\mu_{N_i} \rangle|^2 = \limsup_{i \in I} |\langle x,C_{N_i}'\mu_{N_i} \rangle|^2
		\end{align*}
	Using compactness and convexity of the state space $\mathrm{S}(A)$, we may also assume that the weak* limit $\mu \coloneqq \lim_{i \in I} C_{N_i}'\mu_i\in \mathrm{S}(A)$ exists. Thus, we obtain
		\begin{align*}
			\limsup_{i} \left|\left\langle \frac{1}{N_i} \sum_{n=0}^{N_i-1} S^n x, \mu_i\right\rangle\right|^2 = |\langle x, \mu\rangle|^2.
		\end{align*}	
		It is clear that the state $\mu$ is invariant, i.e., $S'\mu = \mu$. As in the GNS-construction (see, e.g., \cite[Section 3.2]{Pede2018}) the map 
			\begin{align*}
				(\cdot |\cdot)_{\mu} \colon A \times A \rightarrow \C, \quad (y,z) \mapsto \langle z^*y , \mu \rangle
			\end{align*}
		is a positive sesquilinear form yielding a Hilbert space $\mathcal{H}_\mu$: Take the subspace $A_\mu \coloneqq \{y \in A\mid (y|y)_{\mu} = 0\}$ and the completion $\mathcal{H}_\mu$ of the quotient $A/A_\mu$ with respect to the induced norm $\|\cdot\|_\mu$ given by $\|[y]\|_\mu = (y|y)_{\mu}$ for $y \in A$. 
		Since $\mu$ is invariant and $S$ satisfies the Schwarz inequality, $S$ induces a contraction $S_\mu \in \mathscr{L}(\mathcal{H}_\mu)$ with $S_\mu [y] = [S y]$ for every $[y] \in A/A_\mu$. By the mean ergodic theorem the sequence $(\frac{1}{J}\sum_{j=1}^{J}(S_\mu)^j)_{J \in \N}$ converges strongly to the orthogonal projection $P_\mu  \in \mathscr{L}(\mathcal{H}_\mu)$ onto the fixed space $\fix(S_\mu)$ of the operator $S_\mu$ (see \cite[Theorem 8.6]{EFHN2015}). This yields
			\begin{align*}
				|\langle x,\mu\rangle| &= \lim_{J \rightarrow \infty} \left|\left\langle\frac{1}{J} \sum_{j=1}^{J} S^j x, \mu \right\rangle\right|  = \lim_{J \rightarrow \infty} \left|\left(\left[\frac{1}{J}\sum_{j=1}^{J}S^jx\right]\mmid[\mathbbm{1}]\right)_\mu\right| = |(P_\mu [x]| [\mathbbm{1}])_\mu|.
			\end{align*}
		Thus, by the Cauchy--Schwarz inequality and the fact that $\mu$ is unital, we obtain
			\begin{align*}
				|\langle x,\mu\rangle|^2 &\leq \|P_\mu [x]\|_\mu^2 \cdot \|[\mathbbm{1}]\|_{\mu}^2 = \|P_\mu [x]\|_\mu^2 = ( [x]|P_\mu[x])_{\mu} = \lim_{J \rightarrow \infty} \frac{1}{J}\sum_{j=1}^{J} (x|S^j x)_\mu \\
				&= \lim_{J \rightarrow \infty} \frac{1}{J}\sum_{j=1}^{J} \lim_{i \in I} (C_{N_i}'\mu_i)((S^jx)^* x) = \lim_{J \rightarrow \infty} \frac{1}{J}\sum_{j=1}^{J} \lim_{i \in I}\frac{1}{N_i} \sum_{n=0}^{N_i-1}\langle S^n((S^jx)^* x), \mu_i\rangle.
			\end{align*}
		Therefore,
			\begin{align*}
					|\langle x, \mu\rangle|^2 \leq \liminf_{J \rightarrow \infty} \frac{1}{J}\sum_{j=1}^{J} \limsup_{i}\frac{1}{N_i}  \sum_{n= 0}^{N_i -1}  \Re \langle S^n((S^jx)^* x),\mu_i\rangle.
			\end{align*}
\end{proof}
	Using the   reformulation   at the beginning of this section, we obtain the scalar van der Corput inequality \cref{vdc1} as a special case.
	\begin{corollary}[van der Corput for complex numbers]\label{vdcscalar}
		For every bounded sequence $(u_n)_{n \in \N}$ of complex numbers the inequality
		\begin{align*}
			\limsup_{N \rightarrow \infty} \left|\frac{1}{N} \sum_{n=1}^N u_n\right|^2 \leq \liminf_{J \rightarrow \infty} \frac{1}{J}\sum_{j=1}^J \limsup_{N \rightarrow \infty} \frac{1}{N}\sum_{n=1}^{N} \Re (u_n \overline{u_{n+j}})
		\end{align*}
		holds. 
	\end{corollary}
	\begin{proof}
		Apply \cref{vdc2} with $A= \ell^\infty(\N)$, $S \in \mathscr{L}(A)$ the shift given by $S(v_n)_{n \in \N} \coloneqq (v_{n+1})_{n \in \N}$ for $(v_n)_{n \in \N} \in \ell^\infty(\N)$, $(N_i)_{i \in I} = (N)_{n \in \N}$, $\mu_i = \delta_1\colon \ell^\infty(\N) \rightarrow \C, \, (v_n)_{n \in \N} \mapsto v_1$  for every $i \in \N$ and $x = (u_n)_{n \in \N} \in \ell^\infty(\N)$.
	\end{proof}
	However, \cref{vdc2} can also be applied to non-commutative $\mathrm{C}^*$-algebras, e.g., the $\mathrm{C}^*$-algebra $\mathscr{L}(H)$ of all bounded operators on a Hilbert space $H$.
	\begin{corollary}[van der Corput for operators]\label{vdcop}
		Let $H$ be a Hilbert space, $S \in \mathscr{L}(H)$ an isometry and $\xi \in H$ with $\|\xi\| = 1$. Then the inequality
			\begin{align*}
				\limsup_{N \rightarrow \infty} \left| \frac{1}{N}\sum_{n=0}^{N-1} (TS^n\xi|S^n\xi)\right|^2 \leq \liminf_{J \rightarrow \infty} \frac{1}{J}\sum_{j=1}^J \limsup_{N \rightarrow \infty} \frac{1}{N}\sum_{n=0}^{N-1} \Re (S^jTS^n\xi|TS^{n+j}\xi).
			\end{align*}
		holds for every $T \in \mathscr{L}(H)$.
	\end{corollary}
	\begin{proof}
		We consider the implemented operator $\mathcal{S}$ acting on $\mathscr{L}(H)$ via $\mathcal{S}T \coloneq S^*TS$ for $T \in \mathscr{L}(H)$. Since $S$ is an isometry, this is a Markov--Schwarz operator. Now take the vector state $\mu$ on $\mathscr{L}(H)$ given by $\mu(T) \coloneqq (T\xi|\xi)$ for all $T \in \mathscr{L}(H)$ and set $\mu_i \coloneqq \mu$ for all $i \in \N$. 
		\cref{vdc2} applied to $\mathcal{S}$ yields the desired inequality.
	\end{proof}
	
\section{Van der Corput on Hilbert modules} \label{sec:modules}
  Similarly to our first proof of the scalar van der Corput inequality, also
  the proof of the full inequality \cref{vdchilbertvalued} can be carried out in greater generality 
  on so-called Hilbert modules. As we will see in \cref{unif3}, being able to consider 
  Hilbert modules with other underlying algebras than $\ell^\infty(\N)$, 
  e.g.\ $\uC_\ub(\N\times\T)$, gives additional control that allows to establish 
  uniform versions of the van der Corput inequality. These are often used for
  establishing Wiener-Wintner results.
  
  \begin{definition}
			Let $A$ be a unital $\mathrm{C}^*$-algebra. A \emph{pre-Hilbert module} $E$ over $A$ is a   unital\footnote{A left module $E$ over $A$ is \emph{unital} if  $1 \cdot x = x$ for every $x \in E$.}   left module $E$ over the algebra $A$ together with a map 
				\begin{align*}
					(\cdot|\cdot)_A \colon E \rightarrow A, \quad (x,y) \mapsto (x|y)_A
				\end{align*}
			such that
				\begin{enumerate}[(i)]
					\item $(x|x)_A \geq 0$ and $(x|x)_A = 0$ if and only if $x = 0$ for $x \in E$.
					\item $(x|y)_A^* = (y|x)_A$ for all $x,y \in E$.
					\item $(\cdot|y)_A \colon E \rightarrow A$ is an $A$-linear map for every $y \in E$.
				\end{enumerate}
		\end{definition}
		An introduction to this concept can be found in \cite{Lanc1995}. 
		
			We will henceforth assume that $A$ is commutative\footnote{By the Gelfand-Naimark theorem one can (and should) think of $A = \mathrm{C}(K)$ for a compact space $K$.} since in this case the Cauchy--Schwarz inequality 
				\begin{align*}
					|(x|y)_A| \leq (x|x)_A^{\frac{1}{2}} \cdot (y|y)_A^{\frac{1}{2}}
				\end{align*}
			holds for all $x,y \in E$, where modulus and square root are defined via continuous functional calculus (use the Gelfand-Naimark representation of $A$ and apply the Cauchy-Schwarz inequality for semi-inner products pointwise, see \cite[page 49]{DuGi1983}). Setting 
				\begin{align*}
					|x|_A \coloneqq (x|x)_A^{\frac{1}{2}} \textrm{ for } x \in E,
				\end{align*}
			we then obtain a \enquote{vector-valued norm} $|\cdot| \colon E \rightarrow A_+$. This also yields a real-valued norm $\|\cdot\|$ via $\|x\|\coloneqq \||x|_A\|$ for $x \in E$. If $E$ is complete with respect to this norm, then $E$ is called a \emph{Hilbert module} (or a \emph{Hilbert $\mathrm{C}^*$-module}) over $A$. However, completeness is not needed in the following.
		\begin{example}\label{exmodule}
			Let $H$ be a Hilbert space. By defining multiplication componentwise we turn $\ell^\infty(\N, H)$ into a unitary module over $\ell^\infty(\N)$. Setting $(u|v)_{\ell^\infty(\N)}\coloneqq ((u_n|v_n))_{n \in \N}$ for $u=(u_n)_{n \in \N}, v= (v_n)_{n \in \N} \in \ell^\infty(\N, H)$, we arrive at a Hilbert module over $\ell^\infty(\N)$. 
		\end{example}
		For stating a van der Corput inequality on such (pre-)Hilbert modules we now need two operators: A Markov operator on $S$ on the $\mathrm{C}^*$-algebra $A$, and an operator $T \in \mathscr{L}(E)$ on the pre-Hilbert module $E$ which is dominated by $S$ in the following sense.
			\begin{definition}
				Let $E$ be a pre-Hilbert module over a commutative unital $\mathrm{C}^*$-algebra $A$ and $S \in \mathscr{L}(A)$ a Markov operator. A bounded operator $T \in \mathscr{L}(E)$ is \emph{$S$-dominated} if $|Tx|_A^2 \leq S|x|_A^2$ for all $x \in E$.
			\end{definition}
			\begin{example}\label{exdominated}
				Let $H$ be a Hilbert space and consider the Hilbert module $\ell^\infty(\N, H)$ over $\ell^\infty(\N)$. Then the shift $T \in \mathscr{L}(\ell^\infty(\N, H))$ defined by $T(u_n)_{n \in \N}\coloneqq (u_{n+1})_{n \in \N}$ for $(u_n)_{n \in \N} \in \ell^\infty(\N, H)$ is dominated by the shift $S \in \mathscr{L}(\ell^\infty(\N))$.
			\end{example}
		Using similar methods as in the previous section, we obtain the following version of the van der Corput inequality.
		\begin{theorem}\label{vdc3}
				Let $S \in \mathscr{L}(A)$ be a Markov operator on a commutative unital $\mathrm{C}^*$-algebra $A$ and let $T \in \mathscr{L}(E)$ be an $S$-dominated operator on a pre-Hilbert module $E$ over $A$.  Moreover, let $x \in E$, $(N_i)_{i \in I}$ a subnet of $\N$ and $\mu_i \in \mathrm{S}(A)$ a state on $A$ for every $i \in I$. Then
			\begin{align*}
				\limsup_{i} \left\langle \left|\frac{1}{N_i} \sum_{n=0}^{N_i-1} T^n x\right|_A^2, \mu_i\right\rangle \leq \liminf_{J \rightarrow \infty} \frac{1}{J}\sum_{j=1}^{J} \limsup_{i} \frac{1}{N_i} \sum_{n=0}^{N_i-1} \Re \langle S^n(x|T^jx)_A, \mu_i\rangle.
			\end{align*}
		\end{theorem}
	
	\begin{proof}
		We take the Cesàro means $C_N \coloneqq \frac{1}{N} \sum_{n=0}^{N-1}T^n$ for every $N \in \N$. For each $i \in I$ we consider the positive semi-definite sesquilinear form
			\begin{align*}
				&\varphi_i \coloneqq \mu_i \circ (\cdot|\cdot)_A \circ  (C_{N_i} \times C_{N_i}) \colon E \times E \rightarrow \C, \\
				& (y,z)  \mapsto \left\langle \left(\frac{1}{N_i}\sum_{n=0}^{N_i-1}T^n y \mmid\frac{1}{N_i}\sum_{n=0}^{N_i-1}T^n z \right),\mu_i \right\rangle
			\end{align*}
		on $E$. As above, we may then assume that
			\begin{align*}
				\limsup_{i} \left\langle \left|\frac{1}{N_i} \sum_{n=0}^{N_i-1} T^n x\right|_A^2, \mu_i\right\rangle = \limsup_{i} \varphi_i(x,x) = \lim_{i} \varphi_i(x,x).
			\end{align*} 
		 Using that $T$ is $S$-dominated, we obtain that
			\begin{align*}
				|\varphi_i(y,y)| \leq \left\langle \left(\frac{1}{N_i}\sum_{n=0}^{N-1} |T^ny|_A\right)^2 , \mu_i \right\rangle \leq  \left\langle \frac{1}{N_i}\sum_{n=0}^{N-1}|T^ny|_A^2 , \mu_i \right\rangle \leq \left\langle |y|_A^2, \frac{1}{N_i}\sum_{n=0}^{N_i-1}(S^n)'\mu_i\right\rangle
			\end{align*}
		 for $y \in E$ and consequently
		 \begin{align}\label{ineq4}
		 	|\varphi_i(y,z)| \leq \left(\left\langle |y|_A^2, \frac{1}{N_i}\sum_{n=0}^{N_i-1}(S^n)'\mu_i\right\rangle\right)^{\frac{1}{2}}  \cdot \left(\left\langle |z|_A^2, \frac{1}{N_i}\sum_{n=0}^{N_i-1}(S^n)'\mu_i\right\rangle\right)^{\frac{1}{2}} 
		 \end{align}
		for all $y,z \in E$ and $i \in I$ by the Cauchy-Schwarz inequality. In particular, we obtain that $|\phi_i(y,z)| \leq \|y\| \cdot \|z\|$ for all $y,z \in E$ and $i \in I$. We may therefore assume (by passing to a subnet using Tychonoff's theorem) that there is a positive sesquilinear form $\varphi \colon E \times E \rightarrow \C$ with $\lim_{i} \varphi_i(y,z) = \varphi(y,z)$ for all $y,z \in E$. Thus the left-hand side of the desired inequality is given by $\varphi(x,x)$. 
		 Finally, again passing to a subnet, we can assume that (as in the proof of \cref{vdc2}) the weak*-limit $\mu \coloneqq \lim_{i} (\frac{1}{N_i}\sum_{n=0}^{N_i-1}(S^n)')\mu_i$ exists in the state space $\mathrm{S}(A)$. Inequality (\ref{ineq4}) then implies
			\begin{enumerate}[(i)]
				\item $|\varphi(y,z)| \leq (\langle |y|_A^2,\mu \rangle)^\frac{1}{2} \cdot (\langle |z|_A^2,\mu \rangle)^\frac{1}{2}$ for all $y,z \in E$.
			\end{enumerate}
		We also note that $\varphi$ is invariant, i.e.,
			\begin{enumerate}[(i)]
				\setcounter{enumi}{1}
				\item $\varphi \circ (T^n \times T^m) = \varphi$ for all $n,m \in \N_0$.
			\end{enumerate}
		As in the proof of \cref{vdc2} we now construct a contraction on a Hilbert space. To do so, consider the positive semi-definite sesquilinear form 
			\begin{align*}
				(\cdot |\cdot)_{\mu} \colon E \times E \rightarrow \C, \quad (y,z) \mapsto \langle (y|z)_A , \mu \rangle
			\end{align*}
		and the subspace $E_\mu \coloneqq \{x \in E\mid (x|x)_\mu = 0\}$. As in the discussion in \cref{sec:basic}, we write $\mathcal{H}_\mu$ for the Hilbert space constructed as the completion of the quotient $E/E_\mu$ with respect to the induced norm $\|\cdot\|_\mu$. Since $T$ is $S$-dominated, it induces a contraction $T_\mu \in \mathscr{L}(\mathcal{H}_\mu)$ with $T_\mu [y] = [Ty]$ for every $y \in E$. Denote the corresponding mean ergodic projection by $P_\mu \in \mathscr{L}(\mathcal{H}_\mu)$. Moreover, by (i) $\varphi$ defines a bounded positive semi-definite sesquilinear form $\varphi_\mu$ on $\mathcal{H}_\mu$. Using (i) and (ii) we now obtain
			\begin{align*}
				\varphi(x,x) &= \varphi_\mu(P_\mu[x],P_\mu[x]) \leq \|P_\mu[x]\|_\mu^2 
				= ([x],P_\mu[x])_\mu = \lim_{J \rightarrow \infty} \frac{1}{J}\sum_{j=1}^{J} (x|T^jx)_\mu \\
				&= \lim_{J \rightarrow \infty} \frac{1}{J}\sum_{j=1}^{J} \lim_i \frac{1}{N_i}\sum_{n=0}^{N_i-1}\langle S^n(x|T^jx)_A,\mu_i\rangle.
			\end{align*}
		This implies the claim.
	\end{proof}
	\begin{remark}
		\cref{vdc2} and \cref{vdc3} are related as follows: If $S \in \mathscr{L}(A)$ is a Markov operator on a commutative unital $\mathrm{C}^*$-algebra $A$, then, with $T=S$ and $(x|y)_A \coloneqq y^*x$ for $x,y \in A$ in \cref{vdc3}, the right-hand sides of the two inequalities coincide. While \cref{vdc2} gives an estimate of
			\begin{align*}
				\limsup_{i} \left|\left\langle \frac{1}{N_i} \sum_{n=0}^{N_i-1} S^n x, \mu_i\right\rangle\right|^2,
			\end{align*}
		\cref{vdc3} only provides an upper bound for
			\begin{align*}
				\limsup_{i} \left\langle \left(\frac{1}{N_i} \sum_{n=0}^{N_i-1} S^n x\right)^*\left(\frac{1}{N_i} \sum_{n=0}^{N_i-1} S^n x\right), \mu_i\right\rangle.
			\end{align*}
		Therefore, \cref{vdc3} is slightly weaker than \cref{vdc2} in the $\mathrm{C}^*$-algebra setting (however, if $\mu_i$ is multiplicative for every $i \in I$, then both estimates are the same).
	\end{remark}


\section{From operators to sequences}\label{applications}
	We now apply our abstract operator theoretic inequality to derive more concrete versions of the van der Corput lemma. We start with the Hilbert space version of \cref{vdcscalar} which is a direct consequence of \cref{vdc3}.
	\begin{corollary}[van der Corput for Hilbert spaces]\label{vdchilbert}
		For every bounded sequence $(u_n)_{n \in \N}$ in a Hilbert space $H$ the inequality
		\begin{align*}
			\limsup_{N \rightarrow \infty} \left\|\frac{1}{N} \sum_{n=1}^N u_n\right\|^2 \leq \liminf_{J \rightarrow \infty} \frac{1}{J}\sum_{j=1}^J \limsup_{N \rightarrow \infty} \frac{1}{N}\sum_{n=1}^{N} \Re(u_n|u_{n+j})
		\end{align*}
		holds. 
	\end{corollary}
	\begin{proof}
		As in the proof of \cref{vdcscalar} we take $A = \ell^\infty(\N)$, $S \in \mathscr{L}(A)$, $(N_i)_{i \in I} = (N)_{N \in \N}$ and $(\mu_i)_{i \in I} = (\delta_1)_{n \in \N}$. Moreover,  we consider the Hilbert module $E = \ell^\infty(\N, H)$ over $\ell^\infty(\N)$ (see \cref{exmodule}) and let $T$ be the shift on $E$ which is $S$-dominated (see \cref{exdominated}). Then \cref{vdc3} yields the claim.
	\end{proof}
	By considering nets $(\mu_i)_{i \in I}$ of different states $\mu_i$ in \cref{vdc3} we also obtain the following   versions   of the van der Corput inequality (compare with \cite[Lemma 3.18]{KerrLi2016}, \cite[Exercise 20.8]{EFHN2015} and \cite[Proposition 6]{More2018}, respectively).
	\begin{corollary}\label{unif1}
		For every bounded sequence $(u_n)_{n \in \N}$ in a Hilbert space $H$ the inequality
		\begin{align*}
			\limsup_{N \rightarrow \infty} \sup_{M \in \N}\left\|\frac{1}{N} \sum_{n=M+1}^{M+N} u_n\right\|^2 \leq \liminf_{J \rightarrow \infty} \frac{1}{J}\sum_{j=1}^J \limsup_{N \rightarrow \infty} \sup_{M \in \N}\frac{1}{N}\sum_{n=M+1}^{N+M} \Re (u_n|u_{n+j})
		\end{align*}
		holds. 
	\end{corollary}
	\begin{proof}
		We find a sequence $(M_N)_{N \in \N}$ with 
			\begin{align*}
				\limsup_{N \rightarrow \infty} \sup_{M \in \N}\left\|\frac{1}{N} \sum_{n=M+1}^{M+N} u_n\right\|^2 = \limsup_{N \rightarrow \infty} \left\|\frac{1}{N} \sum_{n=M_N+1}^{M_N+N} u_n\right\|^2.
			\end{align*}
		Consider the states $\mu_i \coloneqq \delta_{M_i}\colon \ell^\infty(\N) \rightarrow \C, \, (v_n)_{n \in \N} \mapsto v_{M_i}$ and proceed as in   the   proof of \cref{vdchilbert}. This yields the desired inequality. 
	\end{proof}
	\begin{corollary}\label{unif2}
		For every bounded sequence $(u_n)_{n \in \N}$ in a Hilbert space $H$ the inequality
		\begin{align*}
			\limsup_{N,M \rightarrow \infty} \left\|\frac{1}{N} \sum_{n=M+1}^{M+N} u_n\right\|^2 \leq \liminf_{J \rightarrow \infty} \frac{1}{J}\sum_{j=1}^J \limsup_{N,M \rightarrow \infty} \frac{1}{N}\sum_{n=M+1}^{N+M} \Re (u_n|u_{n+j})
		\end{align*}
		holds. 
	\end{corollary}
	\begin{proof}
		We find subsequences $(N_i)_{i \in \N}$ and $(M_i)_{i \in \N}$ of $\N$ such that 
			\begin{align*}
				\limsup_{N,M \rightarrow \infty} \left\|\frac{1}{N} \sum_{n=M+1}^{M+N} u_n\right\|^2  = \limsup_{i \rightarrow \infty} \left\|\frac{1}{N_i} \sum_{n=M_i+1}^{M_i+N_i} u_n\right\|^2.
			\end{align*}
		We then again proceed as in the proof of \cref{vdchilbert} and apply \cref{vdc3} with $(N_i)_{i \in \N}$ and $(\mu_i)_{i \in \N}$ given by $\mu_i \coloneqq \delta_{M_i}\colon \ell^\infty(\N) \rightarrow \C, \, (v_n)_{n \in \N} \mapsto v_{M_i}$ for $i \in \N$.
	\end{proof}
	\begin{corollary}\label{unif3}
		For every bounded sequence $(u_n)_{n \in \N}$ in a Hilbert space $H$ the inequality
		\begin{align*}
			\limsup_{N \rightarrow \infty} \sup_{|\lambda| = 1}\left\|\frac{1}{N} \sum_{n=1}^N \lambda^nu_n\right\|^2 \leq \liminf_{J \rightarrow \infty} \frac{1}{J}\sum_{j=1}^J \limsup_{N \rightarrow \infty} \left|\frac{1}{N}\sum_{n=1}^{N} (u_n|u_{n+j})\right|
		\end{align*}
		holds. 
	\end{corollary}
	\begin{proof}
		We write $\T \coloneqq \{z \in \C\mid |z| = 1\}$, consider the $\mathrm{C}^*$-algebra $A \coloneqq \mathrm{C}_{\mathrm{b}}(\T \times \N)$ of bounded continuous functions on $\T \times \N$ and the shift operator $S \in \mathscr{L}(A)$ given by $Sf(\lambda,n) \mapsto f(\lambda,n+1)$ for $(\lambda,n) \in \T \times \N$ and $f \in \mathrm{C}_{\mathrm{b}}(\T \times \N)$. Similar to the proof of \cref{vdchilbert} the space of bounded vector-valued continuous functions $E\coloneqq \mathrm{C}_{\mathrm{b}}(\T \times \N,H)$ is canonically a Hilbert module over $A$. The shift $T \in \mathscr{L}(E)$ given by $Ty(\lambda,n) \coloneqq y(\lambda, n+1)$ for $(\lambda,n) \in \T \times \N$ and $y \in \mathrm{C}_{\mathrm{b}}(\T \times \N)$ is $S$-dominated. We set $x(\lambda,n) \coloneqq \lambda^n u_n$ for $(\lambda,n) \in \T \times \N$. Finally, we pick a sequence $(\lambda_N)_{N \in \N}$ in $\T$ with
			\begin{align*}
				\limsup_{N \rightarrow \infty} \sup_{|\lambda| = 1}\left\|\frac{1}{N} \sum_{n=1}^N \lambda^nu_n\right\|^2 = \limsup_{N \rightarrow \infty} \left\|\frac{1}{N} \sum_{n=1}^N \lambda_N^nu_n\right\|^2
			\end{align*}
		and set $\mu_i(f) \coloneqq f(\lambda_i,1)$ for $f \in \mathrm{C}_{\mathrm{b}}(\T \times \N)$ and $i \in \N$. Then by  \cref{vdc3} 
			\begin{align*}
				\limsup_{N \rightarrow \infty} \sup_{|\lambda| = 1}\left\|\frac{1}{N} \sum_{n=1}^N \lambda^nu_n\right\|^2 &= \limsup_{N \rightarrow \infty} \left\langle \left|\frac{1}{N} \sum_{n=0}^{N-1}T^nx\right|_A^2, \mu_N \right\rangle \\
				&\leq \liminf_{J \rightarrow \infty} \frac{1}{J}\sum_{j=1}^{J} \limsup_{N \rightarrow \infty} \frac{1}{N} \sum_{n=0}^{N-1} \Re \langle S^n(x|T^jx)_A, \mu_N\rangle. 
			\end{align*}
		However, for every $j \in \N$, we have
			\begin{align*}
				 \Re \frac{1}{N} \sum_{n=0}^{N-1} \langle S^n(x|T^jx)_A, \mu_N\rangle& \leq \left| \frac{1}{N} \sum_{n=0}^{N-1} (\lambda_N^nu_n|\lambda_N^{n+j}u_{n+j})\right| = \left|\frac{1}{N} \sum_{n=0}^{N-1} (u_n|u_{n+j})\right|
			\end{align*}
		which yields the claim.
	\end{proof}

\section{Generalizations to F{\o}lner nets}
	Our results can easily be generalized from Cesàro means to other ergodic nets. Recall the following definition.
	\begin{definition}
		Let $\EuScript{S}$ be a semigroup (with composition now denoted multiplicatively). A net $(F_i)_{i \in I}$ of non-empty finite subsets of $\EuScript{S}$ is a \emph{right F{\o}lner net} if
			\begin{align*}
				\lim_i \frac{|F_is \Delta F_i|}{|F_i|} = 0
			\end{align*}
		for every $s \in \EuScript{S}$.
	\end{definition}
		\begin{remark}
		 	Note that every abelian semigroup has a (right) F{\o}lner net (see \cite[Theorem 4]{ArWi1967}) and every semigroup having a right F{\o}lner net is necessarily \emph{right amenable}, i.e., has a right invariant mean (however, the converse does not hold, see \cite[Section 4.22]{Pate1988}). 
		\end{remark}
		If the semigroup $\EuScript{S}$ is \emph{right cancellative} (see, e.g., \cite[Definition 1.16]{BeJuMi1989}), i.e., $rt= st$ implies $r=s$ for $r,s,t  \in \EuScript{S}$, then we readily obtain that right F{\o}lner nets induce \enquote{right ergodic operators nets} in the following way (cf. \cite[Definition  1.1]{Schr2013}).
		\begin{lemma}
			Let $(F_i)_{i \in I}$ be a right F{\o}lner net in a right cancellative semigroup $\EuScript{S}$. If $S \colon\EuScript{S} \rightarrow \mathscr{L}(E)$ is a bounded semigroup representation on a Banach space $E$, then the operators $C_i \coloneqq \frac{1}{|F_i|}\sum_{t \in F_i} S_t \in \mathscr{L}(E)$ for $i \in I$ satisfy
				\begin{align*}
					\lim_{i} \|C_iS_t - C_i\| = 0 
				\end{align*}			
			for every $t \in \EuScript{S}$.	 
		\end{lemma}

	Using the same arguments---mutatis mutandis---as in Sections 2 and 3 with the more general version of the mean ergodic theorem for right ergodic operator nets, see \cite[Theorem 8.32]{EFHN2015} and \cite[Theorem 1.7]{Schr2013}, we obtain the following analogues of \cref{vdc2} and \cref{vdc3}.
			\begin{theorem}\label{vdc2b}
		Let $\EuScript{S}$ be a right cancellative semigroup. Moreover, let
			\begin{enumerate}[(i)]
				\item $(F_i)_{i \in I}$ and $(G_j)_{j \in J}$ be right F{\o}lner nets for $\EuScript{S}$,
				\item $A$ be a unital $\mathrm{C}^*$-algebra,
				\item $S \colon \EuScript{S} \rightarrow \mathscr{L}(A)$ a representation as Markov--Schwarz operators, and
				\item $\mu_i \in \mathrm{S}(A)$ be a state for every $i \in I$.  
			\end{enumerate}
		Then
			\begin{align*}
				\limsup_{i} \left|\left\langle \frac{1}{|F_i|} \sum_{t \in F_i} S_t x, \mu_i\right\rangle\right|^2 \leq \liminf_{j} \frac{1}{|G_j|}\sum_{s \in G_j} \limsup_{i} \frac{1}{|F_i|} \sum_{t \in F_i} \Re \langle  S_t((S_sx)^* x) , \mu_i\rangle
			\end{align*}
		for every $x \in A$.
	\end{theorem}
	\begin{theorem}\label{vdc3b}
			Let $\EuScript{S}$ be a right cancellative semigroup. Moreover, let
				\begin{enumerate}[(i)]
				\item $(F_i)_{i \in I}$ and $(G_j)_{j \in J}$ be right F{\o}lner nets for $\EuScript{S}$,
				\item $E$ be a pre-Hilbert module over a unital commutative $\mathrm{C}^*$-algebra $A$,
				\item $S \colon \EuScript{S} \rightarrow \mathscr{L}(A)$ a representation as Markov operators,
				\item $T \colon \EuScript{S} \rightarrow \mathscr{L}(A)$ a representation such that $T_t$ is $S_t$-dominated for every $t \in \EuScript{S}$, and
				\item $\mu_i \in \mathrm{S}(A)$ a state for every $i \in I$.  
			\end{enumerate}
		Then
			\begin{align*}
				\limsup_{i} \left\langle \left|\frac{1}{|F_i|} \sum_{t \in F_i} T_t x\right|_A^2, \mu_i\right\rangle \leq \liminf_{j} \frac{1}{|G_j|}\sum_{s \in G_j} \limsup_{i} \frac{1}{|F_i|} \sum_{t \in F_i} \Re \langle  S_t(x|T_sx)_A , \mu_i\rangle
			\end{align*}
		for every $x \in E$.
		\end{theorem}
	As in Section \ref{applications} these results imply several van   der   Corput inequalities, e.g., the following one.
	\begin{corollary}[van der Corput for semigroups]\label{vdcsgrhilbert}
		Let $\EuScript{S}$ be a right cancellative semigroup with right F{\o}lner nets $(F_i)_{i \in I}$ and $(G_j)_{j \in J}$. For every bounded map $u\colon \EuScript{S} \rightarrow H ,\, t \mapsto u_t $ into a Hilbert space $H$   and every $r \in \EuScript{S}$   the inequality 
			\begin{align*}
				\limsup_{i} \left\|\frac{1}{|F_i|}\sum_{t \in F_i} u_{rt}\right\|^2 \leq \liminf_{j} \frac{1}{|G_j|}\sum_{s \in G_j} \limsup_{i} \frac{1}{|F_i|}\sum_{t \in F_i} \Re(u_{rt}| u_{rts} )
			\end{align*}
		holds.
	\end{corollary}
	\begin{remark}
		 It would be interesting to also apply   our approach to closed subsemigroups $\mathcal{S}$ of a locally compact group $\mathcal{G}$ with   right   Haar measure $m$  having a \emph{topological right F{\o}lner net} $(F_i)_{i \in I}$ (e.g., $\mathcal{S} = \R_{\geq 0}$ and $\mathcal{G} = \R$), i.e., $F_i \subset \mathcal{S}$ is compact with positive measure for every $i \in I$ and
			\begin{align*}
				\lim_i \frac{m(F_i \Delta  F_is)}{m(F_i)} = 0.
			\end{align*}
		for every $s \in \mathcal{S}$ (cf. \cite[Examples 1.2 (e)]{Schr2013}). In particular, this could lead to a new proof of  \cite[Theorem 2.12]{BeMo2016}.
	\end{remark}


\printbibliography

\end{document}